\documentclass[12pt]{amsart}
\usepackage{fullpage}
\usepackage[left=1in,top=1.2in,bottom=1.2in,right=1in]{geometry}
\usepackage[utf8]{inputenc}
\usepackage{tikz}
\usetikzlibrary{shapes.geometric}
\usepackage{cancel}

\usepackage{amsmath, amsthm, tikz, fullpage, longtable,ulem,amssymb}
\usetikzlibrary{graphs, calc, quotes}
\usetikzlibrary{arrows.meta, shadows, fadings,shapes.arrows,positioning}
\tikzset{My Arrow Style/.style={single arrow, draw, text width=0.75cm}}
\usepackage{amsfonts}
\usepackage{MnSymbol,wasysym}
\usepackage{amsmath, amsthm, tikz, fullpage, longtable,ulem,amssymb}
\usepackage{graphicx}
\usepackage{xpatch}
\makeatletter
\newcounter{proofpart}
\xpretocmd{\proof}{\setcounter{proofpart}{0}}{}{}
\newcommand{\proofpart}[1]{%
  \par
  \addvspace{\medskipamount}%
  \stepcounter{proofpart}%
  \noindent\emph{Part \theproofpart: #1}\par\nobreak\smallskip
  \@afterheading
}
\makeatother
\usepackage{caption}
\usepackage{subcaption}
\usepackage{lipsum}
\usepackage{enumerate}
\usepackage{float}
\usepackage{hyperref}

\newcommand{\bc}[1]{\textcolor{blue}{#1}}

\newcommand{\floor}[1]{\lfloor{#1}\rfloor}
\newcommand{\ceil}[1]{\lceil{#1}\rceil}
\usetikzlibrary{graphs,positioning,decorations.pathreplacing}
\newtheorem*{thm*}{Theorem}
\newtheorem{thm}{Theorem}
\newtheorem{obs}[thm]{Observation}

\newtheorem{cor}[thm]{Corollary}
\newtheorem{lemma}[thm]{Lemma}
\newtheorem{defn}[thm]{Definition}
\newtheorem{con}[thm]{Conjecture}

\newtheorem*{state*}{Statement}
\newtheorem{state}[thm]{Statement}

\usepackage{color}
\usepackage{listings}
\usepackage{setspace}

\definecolor{Code}{rgb}{0,0,0}
\definecolor{Decorators}{rgb}{0.5,0.5,0.5}
\definecolor{Numbers}{rgb}{0.5,0,0}
\definecolor{MatchingBrackets}{rgb}{0.25,0.5,0.5}
\definecolor{Keywords}{rgb}{0,0,1}
\definecolor{self}{rgb}{0,0,0}
\definecolor{Strings}{rgb}{0,0.63,0}
\definecolor{Comments}{rgb}{0,0.63,1}
\definecolor{Backquotes}{rgb}{0,0,0}
\definecolor{Classname}{rgb}{0,0,0}
\definecolor{FunctionName}{rgb}{0,0,0}
\definecolor{Operators}{rgb}{0,0,0}
\definecolor{Background}{rgb}{0.98,0.98,0.98}
\lstdefinelanguage{Python}{
numbers=left,
numberstyle=\footnotesize,
numbersep=1em,
xleftmargin=1em,
framextopmargin=2em,
framexbottommargin=2em,
showspaces=false,
showtabs=false,
showstringspaces=false,
frame=l,
tabsize=4,
basicstyle=\ttfamily\small\setstretch{1},
backgroundcolor=\color{Background},
commentstyle=\color{Comments}\slshape,
stringstyle=\color{Strings},
morecomment=[s][\color{Strings}]{"""}{"""},
morecomment=[s][\color{Strings}]{'''}{'''},
morekeywords={import,from,class,def,for,while,if,is,in,elif,else,not,and,or,print,break,continue,return,True,False,None,access,as,,del,except,exec,finally,global,import,lambda,pass,print,raise,try,assert},
keywordstyle={\color{Keywords}\bfseries},
morekeywords={[2]@invariant,pylab,numpy,np,scipy},
keywordstyle={[2]\color{Decorators}\slshape},
emph={self},
emphstyle={\color{self}\slshape},
}

\begin{document}
\thispagestyle{empty}
\title{ Maximal bipartite graphs with a unique minimum dominating set}

\author{Garrison Koch}
\address{Rochester Institute of Technology}
\curraddr{1 Lomb Memorial Dr, Rochester, NY 14623}
\email{glk5534@rit.edu}

\author{Darren Narayan}
\address{Rochester Institute of Technology}
\curraddr{1 Lomb Memorial Dr, Rochester, NY 14623}
\email{dansma@rit.edu}

\begin{abstract}
\noindent  In 2003, Fischermann et al. considered the maximum size of \textit{uniquely-dominatable} graphs, graphs whose dominating number is realized only by a unique dominating set. They conjectured a size bound and provide a general graph construction that shows the bound is tight \cite{Original_Paper}. In 2010, Shank and Fraboni prove Fischermann's bound is true when $\gamma = 2$ \cite{Shank_paper}. In this paper, we observe how Fischermann's bound changes if we impart a different restriction on a graph --- bipartiteness. We conjecture a bound on the maximum number of edges possible for uniquely-dominatable bipartite graphs. We provide constructions to demonstrate this bound is tight. We prove our bipartite bound is true for the $\gamma = 2$ and $n = 3\gamma$ cases. We also discuss perfect domination and how it relates to our extremal graph constructions. We provide constructions that meet both Fischermann's bound for all graphs and our bound for bipartite graphs respectively, both of which are perfectly dominated.
\end{abstract}

\maketitle

\renewcommand{\baselinestretch}{1.2}

\section{Introduction}
A \textit{dominating set} of a finite, simple graph $G = (V,E)$ is a set of vertices $D \subseteq V$ such that every vertex in $V-D$ is adjacent to a vertex in $D$. The domination number, $\gamma(G)$ is the minimum cardinality of a dominating set of $G$. In 1965, Vizing proved that the most number of edges a graph $G$ can have given $|V(G)| = n$ vertices and dominating number $\gamma(G) = \gamma \geq 2$ is $\frac{1}{2}(n-\gamma)(n-\gamma+2)$ \cite{Vizing}. This will henceforth be referred to as \textit{Vizing's bound}. A newer dominating variant studied is called \textit{unique domination}. As the name suggests, a uniquely-dominatable graph is a graph with only one (unique) minimum dominating set. In 2003, Fischermann et al. considered how Vizing's bound could be reduced when they restrict their graphs to only uniquely-dominatable graphs. In their 2003 paper, Fischermann et al. conjecture the upper bound on the size of uniquely-dominatable graphs is
$$\mathbf{m}(n,\gamma) = {{n-\gamma}\choose{2}} - \gamma(\gamma-2)$$ when $n \geq 3\gamma$ and $\gamma \geq 2$ are the order and dominating number respectively \cite{Original_Paper}. We will call this \textit{Fischermann's bound}. To avoid confusion with the subsequent bound we will conjecture, Fischermann's bound will use the above bolded $\mathbf{m}$ when referenced. In their 2003 paper, Fischermann et al. construct uniquely-dominatable graphs with $n$ vertices, domination number $\gamma$, and size $\mathbf{m}(n,\gamma)$ \cite{Original_Paper}. In 2010, Shank and Fraboni proved the above bound is true when $\gamma = 2$ \cite{Shank_paper}. Since then many other graph classes have been considered, such as by Hedetniemi as well as Gunther et al. who characterized how the bound behaves on certain tree graphs \cite{Hedetniemi_Paper} \cite{Gunther}. The unique minimum dominating set problem has also been studied on dominating function variations such as on locating-dominating sets \cite{L-D_Paper}. This leads us to the main conjecture of the paper.
\begin{con}\label{conj:MAINEDGE}
    Let $G$ be a bipartite graph without isolated vertices with order $n$ and dominating number $\gamma$. Let $G$ have a unique minimum dominating set. If $\gamma \geq 2$ and $n\geq 3\gamma$, then the size of $G$, $s(G)$, is bounded \textbf{above} by $$m(n,\gamma) = 2 \gamma + 2 \Big\lceil \frac{\gamma}{2} \Big\rceil \Big\lfloor \frac{\gamma}{2} \Big\rfloor + \min\Big\{n-3\gamma, 2 \Big\lceil \frac{\gamma}{2} \Big\rceil - \Big\lfloor \frac{\gamma}{2}\Big\rfloor+ 1\Big\}\Big(2\Big\lceil\frac{\gamma}{2}\Big\rceil + 1\Big) + \sum_{i=1}^\Phi \Big((2\Big\lceil\frac{\gamma}{2}\Big\rceil+1) + \Big\lceil\frac{i}{2}\Big\rceil\Big)$$
where $\Phi(n,\gamma) = \max(0, n-3\gamma - 2\lceil\frac{\gamma}{2}\rceil - \lfloor\frac{\gamma}{2}\rfloor + 1).$
\end{con}
We will henceforth refer to this as the \textit{bipartite bound}, or simply Conjecture \ref{conj:MAINEDGE}. As we will discuss in Theorem \ref{thm:2privateneighbs}, a unique dominating set of size $\gamma$ is not possible if $n <3\gamma$, hence the restriction seen above. Also, if $\gamma = 1$, the only bipartite graph possible is a $K_{1,n-1}$, which has size $n-1$ and clearly has a unique dominator if $n \geq 3$. In Section \ref{sect:Tightness}, we will provide a generalized bipartite construction to show this bound is tight. In Sections \ref{sect:n3g} and \ref{sect:g=2}, we prove the bound holds for the $n = 3\gamma$ and $\gamma = 2$ cases. In Section \ref{sect:FishConstruction}, we compare the bipartite bound to Fischermann's bound in order to quantify how many edges we lose when we require bipartiteness.
\section{Results}\label{sect:Results}
Let us begin by defining some notation and terminology. Haynes defines the private neighborhood of a vertex $v$ relative to a set of vertices $S$ as $pn(v, S) = \{u \in V : N(u) \cap S = \{v\}\}$ \cite{EPN}. For example, $u$ is a private neighbor of $v$ with respect to the set $S$ if $u$ is adjacent to $v$ and $u$ is not adjacent to any other vertex in $S$. Haynes also defines the exterior private neighborhood of a vertex with respect to a set as $epn(v, S) = pn(v, S) \cap (V - S)$. That is to say, if $u$ is a private neighbor of $v$ with respect to $S$, $u$ is also an exterior private neighbor if $u \notin S$. \cite{Roman_Dom}. Private neighbors will be important for proving the bipartite bound when $\gamma = 2$. They are also pivotal in the construction of a perfectly dominated graphs.
Next, we noted our bipartite bound conjecture is only true for $n \geq 3\gamma$. 
\begin{thm}\label{thm:2privateneighbs}
(From \cite{Original_Paper}) If a graph $G$ of order $n$ without isolated vertices has a unique minimum dominating
set $D$, then it is easy to see that the private exterior neighborhood $epn(v, D) = N(v, G)\setminus N[D\setminus\{v\}, G]$ of $v$ with respect to $D$ contains at least two vertices for each vertex $v \in D$. This observation implies that $n \geq 3\gamma$.
\end{thm}
\begin{proof}
    For the sake of contradiction, assume $v\in D$ has fewer than two exterior private neighbors with respect to the set $D$. Since $G$ contains no isolated vertices, $\exists u \in G$ such that $\{u,v\} \in E(G)$. Then the set $D \cup \{u\} \setminus \{v\}$ is another minimum dominating set, which renders $D$ not unique.
\end{proof}
Thus, throughout this paper, we analyze graphs whose order is at least three times its dominating size.

\subsection{Tightness of Bipartite Bound}\label{sect:Tightness}
In this section, we prove the bound given in Conjecture \ref{conj:MAINEDGE} is tight. To do so, we provide a generalized bipartite construction that is uniquely-dominatable and has the number of edges prescribed by the bound. First, we give a quick lemma that will aid in the final part of Theorem \ref{thm:Constructions}'s proof.
\begin{lemma}\label{lma:SmallerPosInts}
   Let $n$ be a positive integer. Let the number of positive integers smaller than $n$ and of the opposite parity (e.g. $9 \rightarrow \{2,4,6,8\}$) be $\tau(n)$. Then,
   \begin{itemize}
       \item$\tau(n) =  \ceil{\frac{n}{2}} = \frac{n}{2}$ if $n$ is even.
       \item $\tau(n) = \ceil{\frac{n}{2}} - 1 = \floor{\frac{n}{2}}$ if $n$ is odd.
   \end{itemize}
\end{lemma}
\begin{proof}
   The simple induction proof is left to the reader.
\end{proof}
Now, we are ready to prove our bipartite bound is tight.
\begin{thm}\label{thm:Constructions}
 Let $\mathcal{B}_{n,\gamma}$ be the set of all uniquely-dominatable bipartite graphs without isolated vertices with order $n$ and dominating number $\gamma$. Then, for each choice of $n\geq 3\gamma$ and $\gamma \geq 2$, there exists a $G \in \mathcal{B}_{n,\gamma}$ where

    $$s(G) = m(n,\gamma) = 2 \gamma + 2 \Big\lceil \frac{\gamma}{2} \Big\rceil \Big\lfloor \frac{\gamma}{2} \Big\rfloor + \min\Big\{n-3\gamma, 2 \Big\lceil \frac{\gamma}{2} \Big\rceil - \Big\lfloor \frac{\gamma}{2}\Big\rfloor+ 1\Big\}\Big(2\Big\lceil\frac{\gamma}{2}\Big\rceil + 1\Big) + \sum_{i=1}^\Phi \Big((2\Big\lceil\frac{\gamma}{2}\Big\rceil+1) + \Big\lceil\frac{i}{2}\Big\rceil\Big)$$
and $\Phi(n,\gamma) = \max(0, n-3\gamma - 2\lceil\frac{\gamma}{2}\rceil - \lfloor\frac{\gamma}{2}\rfloor + 1).$
\end{thm}
\begin{proof}
    We will break this proof into three cases.\\\\
    \textbf{Case 1:} $n = 3\gamma$.\\
    Notice in this case, the above bound simplifies to $$m(n,\gamma) = 2 \gamma + 2 \lceil \frac{\gamma}{2} \rceil \lfloor \frac{\gamma}{2} \rfloor.$$
    Let $G$ have partite sets $A$ and $B$. Let $A = Y \cup D_X$ and $B = X \cup D_Y$. Let $D_X = \{x_1, x_2, \ldots, x_{\lfloor \frac{\gamma}{2}\rfloor}\}$. Let $X = \{b_{1,1},b_{1,2}, b_{2,1},b_{2,2}, \ldots,b_{\lfloor \frac{\gamma}{2}\rfloor,1},b_{\lfloor \frac{\gamma}{2}\rfloor,2}\}$. Let $D_Y = \{y_1,y_2,\ldots,y_{\ceil{\frac{\gamma}{2}}}\}$. \\Let $Y = \{a_{1,1},a_{1,2}, a_{2,1},a_{2,2}, \ldots,a_{\lceil \frac{\gamma}{2}\rceil,1},a_{\lceil \frac{\gamma}{2}\rceil,2}\}$. Note, $|D_X| = \lfloor \frac{\gamma}{2}\rfloor$, $|D_Y| = \lceil \frac{\gamma}{2}\rceil$, $|X| = 2\lfloor \frac{\gamma}{2}\rfloor$, and $|Y| = 2\lceil \frac{\gamma}{2}\rceil$. So, $|A| = |Y| + |D_X| = \lceil\frac{3\gamma}{2}\rceil = \lceil\frac{n}{2}\rceil$, and, $|B| = |X| + |D_Y| = \lfloor\frac{3\gamma}{2}\rfloor = \lfloor\frac{n}{2}\rfloor$, so $|V(G)|$ indeed equals $3\gamma$.

    Now, let $\{b_{i,1},x_i\}\in E(G)$ and $\{b_{i,2},x_i\}\in E(G)$ for $i = 1,2,\ldots, \floor{\frac{\gamma}{2}}$.  Let $\{a_{j,1},y_j\}\in E(G)$ and $\{a_{j,2},y_j\}\in E(G)$ for $j = 1,2,\ldots, \ceil{\frac{\gamma}{2}}$.
    A generalized depiction of $G$ is given below.
    

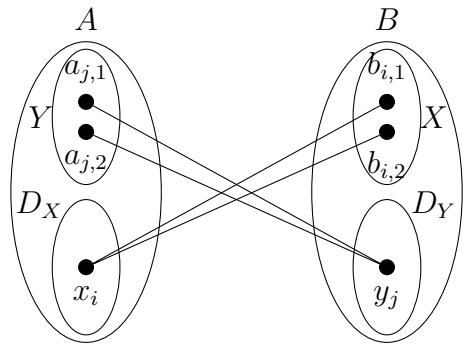
\begin{figure}[H]
\centering
   \begin{tikzpicture}
   \tikzstyle{vertex}=[circle,fill=black,inner sep=2pt]
    \node[ellipse,
    draw = black,
    text = black,
    minimum width = 2cm, 
    minimum height = 4cm,label=above:{$A$}] (e) at (-2,0) {};
    
    \node[ellipse,
    draw = black,
    text = black,
    minimum width = 0.9cm, 
    minimum height = 1.8cm, label={[xshift=-6mm, yshift=-12mm]$Y$}] (e) at (-2,1) {};
     \node[ellipse,
    draw = black,
    text = black,
    minimum width = 0.9cm, 
    minimum height = 1.8cm, label={[xshift=-6.2mm, yshift=-4mm]$D_X$}] (e) at (-2,-1) {};
    
    \node[ellipse,
    draw = black,
    text = black,
    minimum width = 2cm, 
    minimum height = 4cm,,label=above:{$B$}] (e) at (2,0) {};
    \node[ellipse,
    draw = black,
    text = black,
    minimum width = 0.9cm, 
    minimum height = 1.8cm,label={[xshift=6.2mm, yshift=-12mm]$X$}] (e) at (2,1) {};
     \node[ellipse,
    draw = black,
    text = black,
    minimum width = 0.9cm, 
    minimum height = 1.8cm,label={[xshift=6.2mm, yshift=-4mm]$D_Y$}] (e) at (2,-1) {};

  \node [draw, vertex, label = above: $a_{j,1}$] (aj1) at (-2,1.2) {};
\node [draw, vertex, label = below: $a_{j,2}$] (aj2) at (-2,0.8) {};

\node [draw, vertex, label = above: $b_{i,1}$] (bi1) at (2,1.2) {};
\node [draw, vertex, label = below: $b_{i,2}$] (bi2) at (2,0.8) {};

\node [draw, vertex, label = below: $x_i$] (xi) at (-2,-1) {};
\node [draw, vertex, label = below: $y_j$] (yj) at (2,-1) {};

 \draw (bi2) -- (xi) -- (bi1); 
 \draw (aj2) -- (yj) -- (aj1); 
\end{tikzpicture}
\end{figure}
Let us pause and assess $G$ at this stage. Currently, $G$ is a forest of $P_3$'s. Since $|V(G)| = 3\gamma$, $G$ consists of $\gamma$ $P_3$'s. So, $s(G) = 2\gamma$. Since $G$ has $\gamma$ components, the dominating number of $G$ is bounded below by $\gamma$. Let $D = D_X \cup D_Y$. Notice $D$ dominates $G$, as the vertices of $D$ are exactly the center vertices of each $P_3$. Another observation is that $D_X$ dominates $X$ and $D_Y$ dominates $Y$. Finally, $D$ is a unique dominating set as there is only one way to dominate each $P_3$ with only one vertex.


Now, let $\{b_{i,1},a\}\in E(G)$ for each $a \in Y$ and for $i = 1,2,\ldots,\floor{\frac{\gamma}{2}}$. In other words, choose one neighbor of each vertex in $D_X$ to be adjacent to all of $Y$. For simplicity, we have chosen the ``first" neighbor of each vertex in $D_X$ according to our labeling. 
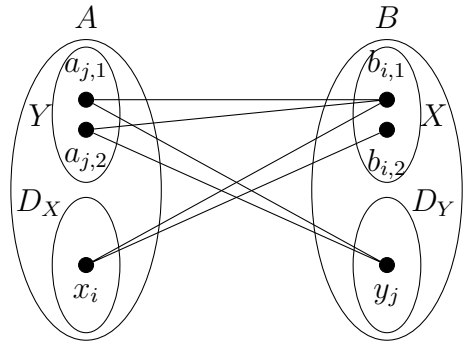
\begin{figure}[H]
\centering
   \begin{tikzpicture}
   \tikzstyle{vertex}=[circle,fill=black,inner sep=2pt]
    \node[ellipse,
    draw = black,
    text = black,
    minimum width = 2cm, 
    minimum height = 4cm,label=above:{$A$}] (e) at (-2,0) {};
    
    \node[ellipse,
    draw = black,
    text = black,
    minimum width = 0.9cm, 
    minimum height = 1.8cm, label={[xshift=-6mm, yshift=-12mm]$Y$}] (e) at (-2,1) {};
     \node[ellipse,
    draw = black,
    text = black,
    minimum width = 0.9cm, 
    minimum height = 1.8cm, label={[xshift=-6.2mm, yshift=-4mm]$D_X$}] (e) at (-2,-1) {};
    
    \node[ellipse,
    draw = black,
    text = black,
    minimum width = 2cm, 
    minimum height = 4cm,,label=above:{$B$}] (e) at (2,0) {};
    \node[ellipse,
    draw = black,
    text = black,
    minimum width = 0.9cm, 
    minimum height = 1.8cm,label={[xshift=6.2mm, yshift=-12mm]$X$}] (e) at (2,1) {};
     \node[ellipse,
    draw = black,
    text = black,
    minimum width = 0.9cm, 
    minimum height = 1.8cm,label={[xshift=6.2mm, yshift=-4mm]$D_Y$}] (e) at (2,-1) {};

  \node [draw, vertex, label = above: $a_{j,1}$] (aj1) at (-2,1.2) {};
\node [draw, vertex, label = below: $a_{j,2}$] (aj2) at (-2,0.8) {};

\node [draw, vertex, label = above: $b_{i,1}$] (bi1) at (2,1.2) {};
\node [draw, vertex, label = below: $b_{i,2}$] (bi2) at (2,0.8) {};

\node [draw, vertex, label = below: $x_i$] (xi) at (-2,-1) {};
\node [draw, vertex, label = below: $y_j$] (yj) at (2,-1) {};

 \draw (bi2) -- (xi) -- (bi1); 
 \draw (aj2) -- (yj) -- (aj1); 
 \draw (aj1) -- (bi1) -- (aj2);
\end{tikzpicture}
\end{figure}

This adds $|Y||D_X| = 2\ceil{\frac{\gamma}{2}}\floor{\frac{\gamma}{2}}$ edges. Thus, $s(G) = 2 \gamma + 2 \lceil \frac{\gamma}{2} \rceil \lfloor \frac{\gamma}{2} \rfloor$, which is our desired size. It remains to show that $D$ is still a UMD set. Consider the following observation:
\begin{obs}\label{obs:NoCommonNeighbs}
    Consider $\{d,d'\} \subseteq D$, two vertices of $D$. Notice that $N[d] \cap N[d'] = \emptyset$. So, each vertex we add to a potential dominating set of $G$ can dominate \textbf{at most one} vertex in $D$. Thus, $\gamma(G) \geq |D| = \gamma$.
\end{obs}
Thus, $D$ remains a minimum dominating set. Finally, we show that $D$ remains unique. By Observation \ref{obs:NoCommonNeighbs}, if we want to build a dominating set $K = \{k_1,k_2,\ldots,k_\gamma\}$ of size $\gamma$, each vertex of $K$ must be a vertex in the closed neighborhood of a different vertex of $D$. That is to say $k_i \in \{b_{i,1},b_{i,2},x_i\}$ for $i = 1,2,\ldots,\floor{\frac{\gamma}{2}}$ and $k_{\overline{j}} \in \{a_{j,1},a_{j,2},y_j\}$ for $\overline{j} = \floor{\frac{\gamma}{2}} + 1 ,\floor{\frac{\gamma}{2}} + 2, \ldots, \gamma$ and $j = \overline{j} - \floor{\frac{\gamma}{2}} = 1,2,\ldots,\ceil{\frac{\gamma}{2}}$ (up to relabeling). Now, for the sake of contradiction assume, for a given $i$, $k_i = b_{i,1}$. Notice, no other vertex in $K - \{k_i\}$ can dominate $b_{i,2}$. Thus, $k_i \not=b_{i,1}$. So, $k_i \in \{b_{i,2},x_i\}$ for $i = 1,2,\ldots,\floor{\frac{\gamma}{2}}$.

Now, for the sake of contradiction, assume $k_{\overline{j}} = a_{j,1}$ for some $\overline{j}$. Notice, no other vertex in $K - \{k_{\overline{j}}\}$ can dominate $a_{j,2}$. So, $k_{\overline{j}} \not=a_{j,1}$. A symmetric argument shows $k_{\overline{j}} \not=a_{j,2}$. 

Thus, we have $k_i \in \{b_{i,2},x_i\}$ and $k_{\overline{j}} = y_j$ (up to relabeling). Finally, if $k_i = b_{i,2}$, $b_{i,1}$ is not dominated by $K$. So, $K = D$. Since we assumed $K$ was an arbitrary dominating set of size $\gamma$, $D$ is unique. This completes this part of the proof.\\\\
\textbf{Case 2:} $3\gamma < n \leq 3\gamma + 2 \lceil \frac{\gamma}{2} \rceil - \lfloor \frac{\gamma}{2}\rfloor+ 1$.\\
Notice in this case, the above bound simplifies to $$m(n,\gamma) = 2 \gamma + 2 \Big\lceil \frac{\gamma}{2} \Big\rceil \Big\lfloor \frac{\gamma}{2} \Big\rfloor + (n-3\gamma)\Big(2\Big\lceil\frac{\gamma}{2}\Big\rceil + 1\Big).$$
Let $G$ be the graph we ended with in case 1. Let $C \subset B$ be an additional set of vertices where $|C| = n - 3\gamma$. Let $\{x_1,c\} \in E(G)$ for each $c \in C$. Let $\{c,a\} \in E(G)$ for each $c \in C$ \textbf{and} $a \in Y$. So, each vertex in $C$ has degree $|Y| + 1 = 2\ceil{\frac{\gamma}{2}} + 1$. Thus, we added $(n-3\gamma)(2\ceil{\frac{\gamma}{2}} + 1)$ edges to $G$, thus $s(G)$ is our desired size. 
Since $C$ is adjacent to $x_1 \in D$, $D$ is still a valid dominating set. Notice, Observation \ref{obs:NoCommonNeighbs} still applies, thus $\gamma(G) \geq |D|$, so $D$ is still a minimum dominating set. It remains to show that $D$ is unique. So, if we want to build a dominating set $K = \{k_1,k_2,\ldots,k_\gamma\}$ of size $\gamma$, each vertex of $K$ must be a vertex in the closed neighborhood of a different vertex of $D$. That is to say, $k_1 \in  \{b_{1,1},b_{1,2},x_1\} \cup C$ and $k_i \in \{b_{i,1},b_{i,2},x_i\}$ for $i = 2,\ldots,\floor{\frac{\gamma}{2}}$ (if such vertices exist, i.e. $\gamma \geq 4$) and $k_{\overline{j}} \in \{a_{j,1},a_{j,2},y_j\}$ for $\overline{j} = \floor{\frac{\gamma}{2}} + 1 ,\floor{\frac{\gamma}{2}} + 2, \ldots, \gamma$ and $j = \overline{j} - \floor{\frac{\gamma}{2}} = 1,2,\ldots,\ceil{\frac{\gamma}{2}}$ (up to relabeling). It follows that $k_1 \notin C$ as $K - k_1$ would not dominate $b_{1,2}$. Thus, $K \cap C = \emptyset$, which tells us that $D$ is still unique. This is because any set that did not dominate $G$ at the end of Case 1 cannot dominate $G$ here, so for another dominating set to exist other than $D$, it must include a vertex in $C$. We showed that $C$ cannot intersect with a minimum dominating set, which concludes the uniqueness argument. This leaves us with one remaining case.\\\\
\textbf{Case 3:} $n >  3\gamma + 2 \lceil \frac{\gamma}{2} \rceil - \lfloor \frac{\gamma}{2}\rfloor+ 1.$\\
Let $G$ be the graph we ended with in case 2 where $|C| = 2 \lceil \frac{\gamma}{2} \rceil - \lfloor \frac{\gamma}{2}\rfloor+ 1$. Additionally, we separate $X$ into $X_1$ and $X_2$. Let $X_1 = \{b_{1,1}, b_{2,1},\ldots\}$ and $X_2 = \{b_{1,2}, b_{2,2}, \ldots\}$. Notice, the vertices in $X_2$ all have degree 1 and are all adjacent to their corresponding vertex in $D_X$. Let $R = \{r_1,r_2,\ldots\}$ be the ``remaining" set of vertices to reach our desired order, where $|R| = n - (3\gamma + 2 \lceil \frac{\gamma}{2} \rceil - \lfloor \frac{\gamma}{2}\rfloor+ 1)$. Let $R = R' \cup R''$, where $R' = \{r_1,r_3,\ldots\}$ and $R''=\{r_2,r_4,\ldots\}$, where $R' \subset A$ and $R'' \subset B$. For reference, a bipartite depiction of the vertices in $G$ are given below.
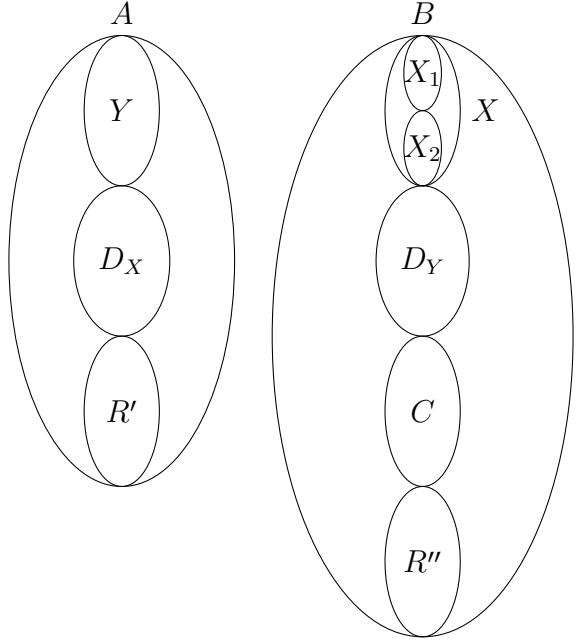
\begin{figure}[H]
\centering
   \begin{tikzpicture}
    \node[ellipse,
    draw = black,
    text = black,
    minimum width = 3cm, 
    minimum height = 6cm,label=above:$A$] (e) at (-2,-1) {};
    
    \node[ellipse,
    draw = black,
    text = black,
    minimum width = 1cm, 
    minimum height = 2cm] (e) at (-2,1) {$Y$};
     \node[ellipse,
    draw = black,
    text = black,
    minimum width = 1cm, 
    minimum height = 2cm] (e) at (-2,-1) {$D_X$};
     \node[ellipse,
    draw = black,
    text = black,
    minimum width = 1cm, 
    minimum height = 2cm] (e) at (-2,-3) {$R'$};
    
    \node[ellipse,
    draw = black,
    text = black,
    minimum width = 4cm, 
    minimum height = 8cm,label=above:$B$] (e) at (2,-2) {};
    \node[ellipse,
    draw = black,
    text = black,
    minimum width = 1cm, 
    minimum height = 2cm,,label=right:$X$] (e) at (2,1) {};
    \node[ellipse,
    draw = black,
    text = black,
    minimum width = 0.5cm, 
    minimum height = 1cm,,label=center:$X_1$] (e) at (2,1.5) {};
    \node[ellipse,
    draw = black,
    text = black,
    minimum width = 0.5cm, 
    minimum height = 1cm,,label=center:$X_2$] (e) at (2,0.5) {};
     \node[ellipse,
    draw = black,
    text = black,
    minimum width = 1cm, 
    minimum height = 2cm] (e) at (2,-1) {$D_Y$};
     \node[ellipse,
    draw = black,
    text = black,
    minimum width = 1cm, 
    minimum height = 2cm] (e) at (2,-3) {$C$};
     \node[ellipse,
    draw = black,
    text = black,
    minimum width = 1cm, 
    minimum height = 2cm,label=center:$R''$] (e) at (2,-5) {};
\end{tikzpicture}
\end{figure}

Now, add the edges 
\begin{itemize}
    \item $\{r',c\}$ for $r' \in R'$ and $c \in C$ (Make a complete bipartite subgraph between $R'$ and $C$).
    \item $\{r',b\}$ for $r' \in R'$ and $b\in X_1$ (Make a complete bipartite subgraph between $R'$ and $X_1$).
    \item $\{r',y_1\}$ for $r' \in R'$.
\end{itemize}
Also add the edges
\begin{itemize}
    \item $\{r'', a\}$ for $r'' \in R''$ and $a \in Y$ (Make a complete bipartite subgraph between $R''$ and $Y$).
    \item $\{r'',x_1\}$ for $r'' \in R''$.
    \item $\{r',r''\}$ for $r' \in R'$ and $r'' \in R''$ (Make a complete bipartite subgraph between $R'$ and $R''$).
\end{itemize}
Notice $|V(G)| = 3\gamma + |C| + |R| = 3\gamma + 2 \lceil \frac{\gamma}{2} \rceil - \lfloor \frac{\gamma}{2}\rfloor+ 1 + (n - (3\gamma +2 \lceil \frac{\gamma}{2} \rceil - \lfloor \frac{\gamma}{2}\rfloor+ 1)) = n$. Now, notice every edge we have added to $G$ has at least one end in $R$ (and many have both ends in $R$). So, it suffices to find the total degree of all the vertices in $R$. So we do not double count edges within $R$, \textbf{when considering the degrees below, we will consider the edge $\{r_i,r_j\}$ where $i < j$ in the neighborhood of $r_j$ and not of $r_i$.} We will refer to these intentional omissions as the \textit{lower degree} of $r_i$
\begin{itemize}
    \item The lower degree of $r_1$ then is $|C| + |X_1| + 1 =$ $2\ceil{\frac{\gamma}{2}}+2$.
    \item The lower degree of $r_2$ is $|Y| + |\{r_1\}| + 1 = 2\ceil{\frac{\gamma}{2}}+2$.
    \item The lower degree of $r_3$ is $|C| + |X_1|+ |\{r_2\}| + 1 = 2\ceil{\frac{\gamma}{2}}+3$.
    \item The lower degree of $r_4$ is $|Y| + |\{r_1,r_3\}| + 1 = 2\ceil{\frac{\gamma}{2}}+3$. And so on.
\end{itemize}
In general, the lower degree of $r_i \in R'$ is $|C| + |X_1| + 1 + \tau(i) = 2\ceil{\frac{\gamma}{2}} + 2 + \ceil{\frac{i}{2}} - 1 = 2\ceil{\frac{\gamma}{2}} + 1 + \ceil{\frac{i}{2}}$ by Lemma \ref{lma:SmallerPosInts}. The lower degree of $r_i \in R''$ is $|Y| + 1 + \tau(i) =  2\ceil{\frac{\gamma}{2}} + 1 + \ceil{\frac{i}{2}}$ by Lemma \ref{lma:SmallerPosInts}. Therefore, we added $$\sum_{i=1}^{|R|}2\ceil{\frac{\gamma}{2}} + 1 + \ceil{\frac{i}{2}}$$ edges to $G$, which indeed equals our prescribed size stated in Theorem \ref{thm:Constructions}. It remains to show that $D$ is a UMD set.\\
Note, Observation \ref{obs:NoCommonNeighbs} still holds, so $D$ is still minimum. Since $D$ was a UMD set at the end of case 2, if another minimum dominating set were to exist here, it would have to include at least one vertex in $R$. If a vertex from $R''$ were to be in a minimum dominating set, it would mean that no other vertex in $N[x_1] = \{x_1,b_{1,1},b_{1,2}\} \cup C \cup R''$ can be in the same minimum dominating set (as this would leave us to find a $\gamma$ - 2 sized set to dominate $ D_Y \cup D_X - \{x_1\}$ which is impossible by Observation \ref{obs:NoCommonNeighbs} and the pigeon-hole principle). This would result in the vertex $b_{1,2}$ left undominated.\\
If a vertex from $R'$ were to be in a minimum dominating set, it would mean that no other vertex in $N[y_1] = \{y_1,a_{1,1},a_{1,2}\} \cup R'$ can be in the same minimum dominating set (as this would leave us to find a $\gamma$ - 2 sized set to dominate $ D_X \cup D_Y - \{y_1\}$ which is impossible by Observation \ref{obs:NoCommonNeighbs} and the pigeon-hole principle). Since vertices in $R'', C,X_1$, and $X_2$ also cannot be in a minimum dominating set, this choice would leave $a_{1,1}$ and $a_{1,2}$ undominated. Thus, vertices of $R$ cannot reside in any minimum dominating set of $G$, and $D$ is unique. This completes the proof, and leaves us with a generalized extremal construction for any choice of $n$ and $\gamma$.
\end{proof}
As is necessary for most generalized constructions, the graphs we create in the proof above are quite patterned. These patterns help us determine the size and minimum dominating set of each graph. However, they also provide a unique dominating set that is \textit{perfect}.
\begin{defn}\label{defn:PerfecrDom}
    Let $G$ be a simple graph without isolated vertices with order $n$ and dominating number $\gamma$. Let $D$ be a minimum dominating set of $G$. If $\sum_{x \in D} deg(x) = n - \gamma$, then $D$ is a perfect dominating set, and $G$ is perfectly dominated by $D$.
\end{defn}
Equivalently, if every vertex is either a dominator or is adjacent to exactly one dominator, and the dominating set forms an independent set, then the graph is perfectly dominated. Notice, this is the case for the general construction provided in the proof above. $D_X$ and $D_Y$ never have edges between each other, and the vertices in $G - \{D_X \cup D_Y\}$ are only adjacent to one dominator. Thus, the following corollary results. Note, a star graph, the only bipartite graph with a dominating number of 1, is a perfectly dominated graph (in fact every graph for which $\gamma(G) = 1$ is perfectly dominated).
\begin{cor}\label{cor:PerfectDomConstruction}
     Let $\mathcal{B}_{n,\gamma}$ be the set of all uniquely-dominatable bipartite graphs without isolated vertices with order $n$ and dominating number $\gamma$. Then, for each choice of $n\geq 3\gamma$ and $\gamma \geq 2$, there exists a perfectly dominated $G \in \mathcal{B}_{n,\gamma}$ where $$s(G) = 2 \gamma + 2 \Big\lceil \frac{\gamma}{2} \Big\rceil \Big\lfloor \frac{\gamma}{2} \Big\rfloor + \min\Big\{n-3\gamma, 2 \Big\lceil \frac{\gamma}{2} \Big\rceil - \Big\lfloor \frac{\gamma}{2}\Big\rfloor+ 1\Big\}\Big(2\Big\lceil\frac{\gamma}{2}\Big\rceil + 1\Big) + \sum_{i=1}^\Phi \Big((2\Big\lceil\frac{\gamma}{2}\Big\rceil+1) + \Big\lceil\frac{i}{2}\Big\rceil\Big)$$
where $\Phi(n,\gamma) = \max(0, n-3\gamma - 2\lceil\frac{\gamma}{2}\rceil - \lfloor\frac{\gamma}{2}\rfloor + 1).$
\end{cor}
\begin{proof}
    The graph construction given the proof of Theorem \ref{thm:Constructions} is perfectly dominated.
\end{proof}

It seems counterintuitive that the most number of edges you can pack into a graph can always be done in such a restrictive manner as a perfect domination. In fact, the \textbf{fewest} number of edges you can have in a graph with order $n$ and minimum dominating set $\gamma$ is $n-\gamma$. So, it would appear that one extra edge between the two partite dominating sets would not cause our uniqueness condition to fail. Analyzing these graphs, it becomes clear that the force behind the uniqueness of this perfect dominating construction lies in the fact that the unique dominating set is impossible to dominate if they are not the dominators themselves, which renders the set unique. Contrary to Fischermann's bound, there are many choices for $n$ and $\gamma$ for which the only construction that meets our bipartite bound is a perfectly dominated graph. The graph in figure \ref{fig:10-3ExGraph} is the only bipartite graph on 10 vertices and 15 edges (the bound) which has a unique minimum dominating set. The graph is also perfectly dominated. We have began to categorize which extremal UMD set graphs have only one construction - a perfectly dominated one - but this topic is beyond the scope of this project.
\begin{figure}[H]
    \centering
    \begin{tikzpicture}
        \tikzstyle{white}=[circle,fill=black,inner sep=2pt]
        \node [draw, white, label=left:$x_1$] (x1) at (0,-4) {};
        \node [draw, white, label=left:$a_{1,1}$] (a11) at (0,0) {};
        \node [draw, white, label=left:$a_{1,2}$] (a12) at (0,-1) {};
         \node [draw, white, label=left:$a_{2,1}$] (a21) at (0,-2) {};
        \node [draw, white, label=left:$a_{2,2}$] (a22) at (0,-3) {};
         \node [draw, white, label=right:$b_{1,1}$] (b11) at (3,0) {};
        \node [draw, white, label=right:$b_{1,2}$] (b12) at (3,-1) {};
        \node [draw, white, label=right:$y_1$] (y1) at (3,-2) {};
        \node [draw, white, label=right:$y_2$] (y2) at (3,-3) {};
        \node [draw, white, label=right:$c_1$] (c1) at (3,-4) {};

         \draw (a11) -- (y1) -- (a12); 
          \draw (a21) -- (y2) -- (a22); 
           \draw (b11) -- (x1) -- (b12); 
            \draw (c1) -- (x1);  
            \draw (a12) -- (c1) -- (a11);
             \draw (a11) -- (b11) -- (a12); 
              \draw (a21) -- (b11) -- (a22); 
                \draw (a21) -- (c1) -- (a22); 
    \end{tikzpicture}
    \caption{The only uniquely dominatable bipartite graph with $n=10$, $\gamma = 3$ and $s = m(10,3) = 15$.}
    \label{fig:10-3ExGraph}
\end{figure}
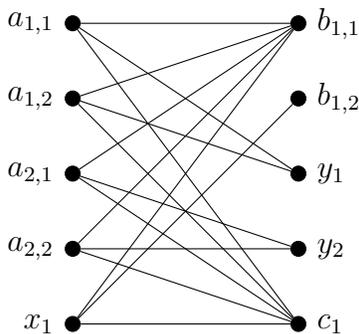
\subsection{The Bipartite Bound when $\gamma(G) = 2$}\label{sect:g=2}
In this section, we prove the bipartite bound is exact when $\gamma(G) = 2$. First, we provide a seemingly unrelated lemma that will aid in this proof.
\begin{lemma}\label{lma:MinEdgesForest}
Let $G$ be a graph with $n \geq 3$ that has no isolated vertices and no $P_2$ components. The \textbf{minimum} number of edges $G$ can have is given by $$s(G) \geq \Big\lceil\frac{2n}{3}\Big\rceil.$$
\end{lemma}
\begin{proof}
    First, note each component must be a tree. For the sake of contradiction, assume that $G_1$ is a connected component of $G$ that is not a tree. Thus, $|V(G_1)| \geq 3$. Let $e$ be a non-cut edge of $G_1$. Then $G_1 - e$ is still a connected component, and is not a $P_2$ (as it has at least three connected vertices). So, $G-e$ has fewer edges than $G$ while retaining the two criteria.\\
    So, $G$ is a forest. Let $G$ have $c$ connected components. The size of $G$ (and every forest) is $s(G) = n - 1 - c$. So, to minimize $s(G)$, we want to maximize $c$, i.e. have the most non-$P_2$ components possible. This is clearly achieved by creating the most $P_3$'s as possible, and appending the final one or two vertices (if they exist) onto any of the existing components. This results in $$s(G) = 2\lfloor\frac{n}{3}\rfloor + \text{mod}(n,3).$$
    It suffices to show that $\lceil\frac{2n}{3}\rceil = 2\lfloor\frac{n}{3}\rfloor + \text{mod}(n,3)$. Let $n = 3k + r$, then $\lceil\frac{2n}{3}\rceil = 2k + r = 2\lfloor\frac{n}{3}\rfloor + \text{mod}(n,3)$
 \end{proof}
 Now, recall the complement of a simple graph $G$, $G^c$. The edges of $G^c$ are exactly the non-edges of $G$. We define a bipartite complement, where instead of all non-edges of $G$, the bipartite complement only consists of all non-edges between the two partite sets.
 \begin{defn}\label{def:NonEdges}
     Let $G$ be a bipartite graph with partite sets $A$ and $B$. The \textbf{Bipartite Complement} of a graph, $\mathbf{BC}(G)$ is the graph with vertex set $A \cup B$ and edge set $\Big\{\{a,b\}\,:\,a\in A, b\in B, $ and $\{a,b\}\notin E(G)\Big\}$. The edges of $\mathbf{BC}(G)$ will be referred to as non-partite-edges of $G$, or simply non-edges if the graph is understood to be bipartite.
 \end{defn}
Next, let us simplify our bipartite bound when $\gamma = 2$ is plugged in.
\begin{state}\label{state:g2simplified}
    When $\gamma = 2$, the bound given in Conjecture \ref{conj:MAINEDGE} simplifies to
    $$m(n,2) = \begin{cases}
        \frac{n(n-2)}{4}, \text{     if $n$ is even}\\
        \frac{(n-1)^2}{4},\text{     if $n$ is odd.}
    \end{cases}$$
\end{state}
\begin{proof}
    To begin our simplification, notice $m(6,2) = 6$ and $m(7,2) = 9$ when plugged directly into the bipartite bound. Now, when $n \geq 8$, the bipartite bound becomes 
    \begin{align}\label{eq:Semi-Simple}
        m(n,2) &= 12 + \sum_{i=1}^{n-8} (3 + \ceil{\frac{i}{2}})\\
        &= 12 + 3(n-8) + \bc{\sum_{i=1}^{n-8} \ceil{\frac{i}{2}}}.
        \label{eq:HighlightedSum}
    \end{align}
Let us simplify the highlighted summation. If $n$ is even, let $2k = n-8$. Then, the summation becomes
\begin{align*}
    \sum_{i=1}^{2k}\ceil{\frac{i}{2}} &= (1 + 1 + 2 + 2 + \cdots + k + k)\\
    &= 2(1 + 2 + \cdots + k)\\
    &= 2 {k+1\choose 2}\\
    &= 2 \frac{(k+1)(k)}{2}\\
    &= \frac{(n-6)(n-8)}{4}
\end{align*}
after plugging $k = \frac{n-8}{2}$ back in. Plugging this result into equation \ref{eq:Semi-Simple}, we get
\begin{align*}
    m(n,2) &= 12 + 3(n-8) + \frac{(n-6)(n-8)}{4}\\
    &= \frac{\cancel{48} + 12n \cancel{- 96} + n^2 - 14n \cancel{+ 48}}{4}\\
   m(n,2) &= \frac{n(n-2)}{4} \text{  when } n \text{ is even.}
\end{align*}
If $n$ is odd, let $2k + 1 = n - 8$. Then, the highlighted summation in equation \ref{eq:HighlightedSum} becomes
\begin{align*}
     \sum_{i=1}^{2k}\ceil{\frac{i}{2}} &= (1 + 1 + 2 + 2 + \cdots + k + k + (k+1))\\
     &= {k + 1\choose 2} + {k+2\choose 2}\\
     &= \frac{2k^2 + 4k + 2}{2}\\
     &= \frac{n^2 - 14n + 49}{4}
\end{align*}
after plugging $k = \frac{n-9}{2}$ back in. Plugging this result into equation \ref{eq:Semi-Simple}, we get
\begin{align*}
    m(n,2) &= 12 + 3(n-8) + \frac{n^2 - 14n + 49}{4}\\
    &= \frac{n^2 +12n - 14n + 49 - 48}{4}\\
   m(n,2) &= \frac{(n-1)^2}{4}, \text{ when } n \text{ is odd.}
\end{align*}
Finally, we see that $m(6,2) = \frac{6(4)}{4} = 6$ and $m(7,2) = \frac{6^2}{4} = 9$ both correctly plug into their respective equations above. So, the simplification is correct for all $n \geq 3\gamma$, completing the proof.
\end{proof}
Now are ready to show the bipartite bound is true when $\gamma = 2$.
\begin{thm}
    Let $G_{n,2} = G$ be a bipartite graph without isolated vertices with order $n \geq 6$, $\gamma = 2$, and a unique dominating set. Then, the size of $G$ is bounded above by $$m(n,2) = m(G) = \begin{cases}
        \frac{n(n-2)}{4}, \text{     if $n$ is even}\\
        \frac{(n-1)^2}{4},\text{     if $n$ is odd.}
    \end{cases}$$
\end{thm}
\begin{proof}
    Let $D = \{x,y\}$. Throughout this proof, let $D \subseteq V(G)$ be the unique dominators of $G$.  By Theorem \ref{thm:2privateneighbs}, in every case of this proof, it must be true that $|epn(x,D)| \geq 2$ and $|epn(y,D)| \geq 2$.
    
    So, there are three cases we must consider:
\proofpart{$x$ and $y$ are in the same partite set.}
Let $A$ and $B$ be the partite sets of $G$. Let $D = \{x,y\} \in A$. Since $D$ dominates $G$, then $A$ must consist only of $\{x,y\}$, as any other vertex in $A$ would not be adjacent to $x$ or $y$. So, $|A| = 2$, and $|B| = (n-2)$. In this case, a complete bipartite graph would have $2(n-2)$ edges. However, we know that $|epn(x,D)| \geq 2$ and $|epn(y,D)| \geq 2$. So, there are at least 4 edges that cannot exist between $A$ and $B$. Thus, when $x$ and $y$ are in the same partite set the size of $G$ is bounded above by $$m_1(G) = 2(n-2) - 4.$$ It remains to show that $m_1(G) \leq m(G).$\\
Let $n$ be even. Let $P_{1e}(n) = m(G) - m_1(G)$. We aim to show $P_{1e}(n) \geq 0$ when $n \geq 6$ (and even). Notice, the polynomial simplifies to \begin{align*}
P_{1e}(n) &= n^2 - 10n + 32  \\  
&= n^2 - 10n +25 + 7\\
&= (n-5)^2 + 7
\end{align*}
after completing the square. It is obvious from here that $P_{1e}(n) \geq 0$ for $n \geq 6$.\\
Similarly, let $n$ be odd.  Let $P_{1o}(n) = m(G) - m_1(G)$. We aim to show $P_{1o}(n) \geq 0$ when $n \geq 6$ (and odd). Notice, the polynomial simplifies to \begin{align*}
P_{1o}(n) &= n^2 - 10n + 33  \\  
&= n^2 - 10n +25 + 8\\
&= (n-5)^2 + 8
\end{align*}
after completing the square. It is obvious from here that $P_{1o}(n) \geq 0$ for $n \geq 6$.\\
\proofpart{$x$ and $y$ are adjacent.}
Let $A$ and $B$ be the partite sets of $G$. Let $x \in A$, and $y \in B$. For simplicity, let $X = \{b_1,b_2,\ldots\} = B - \{y\}$ and $Y = \{a_1,a_2,\ldots\} = A - \{x\}$. So, $X \subset N(x)$ and $Y \subset N(y)$. Now, there are two conditions that the vertices in $X$ and $Y$ must meet:\\
\textbf{Condition 1:} For every $a \in Y$ there must exist a $b \in X$ such that $\{a,b\} \notin E(G)$. Similarly, for every $b \in X$ there must exist a $a \in Y$ such that $\{b,a\} \notin E(G)$. It follows that $\deg(a) \leq |X|$ and $\deg(b) \leq |Y|$.\\
\textbf{Proof of Condition 1:} For the sake of contradiction, Let $a \in Y$ be adjacent to all of $X$. Then, $\{y,a\}$ dominate $\{x,y,X,Y\} = V(G)$, thus $D$ is not unique. Let $b \in X$ be adjacent to all of $Y$. Then, $\{x,b\}$ dominate $\{x,y,X,Y\} = V(G)$, thus $D$ is not unique.\\
\textbf{Condition 2:} For every $a \in Y$ and $b \in X$ for which $\{a,b\} \notin E(G)$, there must also exist an $a' \in Y$  or a $b' \in X$ such that $\{a,b'\} \notin E(G)$ or $\{a',b\} \notin E(G)$.\\
\textbf{Proof of Condition 2:} For the sake of contradiction, Let $a \in Y$ be adjacent to all of $X - b$, and $b \in X$ be adjacent to all of $Y - a$. Then, $\{a,b\}$ dominate $\{x,y,X,Y\}$, and $D$ is not unique.

So, we want to consider the maximum number of edges $G$ can have while satisfying these conditions. First, notice that $x$ is already adjacent to all of $B$, and $y$ is already adjacent to all of $A$. So, maximizing the edges of $G$ boils down to considering the edges in the induced subgraph $G-D$, i.e. the edges between $X$ and $Y$. Maximizing the edges between $X$ and $Y$ is analogous to minimizing the edges between $X$ and $Y$ in its bipartite complement. Consider the makeup of $\mathbf{BC}(G -D)$. First, there are no isolated vertices by Condition 1. Next, there are no $P_2$'s by Condition 2. Thus, by Lemma \ref{lma:MinEdgesForest} $$s(\mathbf{BC}(G-D)) \geq \bigg\lceil \frac{2(|X| + |Y|)}{3}\bigg\rceil.$$ Thus, $$s(G-D) \leq |X||Y| - \bigg\lceil \frac{2(|X| + |Y|)}{3}\bigg\rceil.$$
Noticing that $\deg(x) = |X| + 1$ and $\deg(y) = |Y| + 1$ (but they share an edge between each other), so the size of $G$ is can be written as 
\begin{align}\label{eq:xyAdjBound}
    s(G) \leq |X| + |Y| + 1 + |X||Y| - \bigg\lceil \frac{2(|X| + |Y|)}{3}\bigg\rceil.
\end{align}
Since the value in the ceiling function is always positive, we can simplify our future arithmetic by dropping the ceiling function (doing so only makes the resulting bound larger). By construction, we know $|X| + |Y| = (n-2)$, however $|X|$ and $|Y|$ can vary. Since $|X| + |Y|$ is a fixed sum, to maximize $|X||Y|$, we want $|X|$ and $|Y|$ to be as close as possible. So, to maximize the right hand side of Equation \ref{eq:xyAdjBound}, $|X||Y| = \lceil\frac{n-2}{2}\rceil\lfloor\frac{n-2}{2}\rfloor$

Thus, an upper bound on the number of edges $G$ can have when $x$ and $y$ are adjacent is given by $$m_2(G) = n - 1 + \Big\lceil\frac{n-2}{2}\Big\rceil\Big\lfloor\frac{n-2}{2}\Big\rfloor - \frac{2(n-2)}{3}.$$
Now, $\lceil\frac{n-2}{2}\rceil\lfloor\frac{n-2}{2}\rfloor$ will behave differently depending on if $n$ is even or odd. If $n$ is even, \begin{align*}
    m_2(G) &= n-1 + \frac{n^2 - 4n + 4}{4} - \frac{2(n-2)}{3}\\
    &= \frac{\cancel{12n}  \cancel{- 12} + 3n^2 \cancel{- 12n} + \cancel{12} - 8n + 8}{12}\\
    m_2(G)&= \frac{3n^2 - 6n - (2n-8)}{12}\label{eq:m_2even}
\end{align*}
If $n$ is odd,
\begin{align*}
     m_2(G) &= n-1 + \frac{n^2 - 4n + 3}{4} - \frac{2(n-2)}{3}\\
    &= \frac{\cancel{12n}  {- 12} + 3n^2 \cancel{- 12n} + {9} - 8n + 8}{12}\\
    &= \frac{3n^2 - 3 - 8n + 8}{12} \\
    m_2(G)&= \frac{3n^2 - 6n + 3 - (2n-2)}{12}.
\end{align*}
So, \begin{align*}
    m_2(G) = \begin{cases}
        \frac{3n^2 - 6n - (2n-8)}{12}, \text{  if $n$ is even}\\
        \frac{3n^2 - 6n + 3 - (2n-2)}{12}, \text{ if $n$ is odd.}
    \end{cases}
\end{align*}
Rewriting $m(G)$ above to have the same denominator, we get
\begin{align*}
    m(G) = \begin{cases}
        \frac{3n^2 - 6n}{12}, \text{     if $n$ is even}\\
        \frac{3n^2 - 6n + 3}{12},\text{     if $n$ is odd.}
    \end{cases}
\end{align*}
Writing $m_2(G)$ and $m(G)$ in this manner, it is clear that $m_2(G) \leq m(G)$ (since $n \geq 6$).\\
\proofpart{$x$ and $y$ are non-adjacent in different partite sets.}The final case is very similar to the case above, however $x$ and $y$ are no longer adjacent. Recall the conditions from Part 2 of this proof. Notice Condition 1 above relaxes without an edge between $x$ and $y$. Now, we need only require every vertex in one set (either $|X|$ or $|Y|$) to have a non-edge to the other partite set. More formally,\\
\textbf{Condition 1:} Let $Q = \mathbf{BC}(G - D)$, where the partite sets of $Q$ remain $X$ and $Y$. Then, in $Q$, either the $\deg(b) \geq 1$ for all $b \in X$ or $\deg(a) \geq 1$ for all $a \in Y$.\\
\textbf{Proof of Condition 1:} For the sake of contradiction, Let $a \in Y$ be adjacent to all of $X$ and $b \in X$ be adjacent to all of $Y$. Then $\{a,b\}$ dominates $\{x,y,X,Y\} = V(G)$, and $D$ is not unique.

It is worth noting, in the proof above, if you remove the edge between $a$ and $b$ they still dominate the graph. So, Condition 2 from Part 2 of this proof still applies here. However, given the change in the first condition, Condition 2 does not cause fewer edges to be possible, so we need not consider it here. 

So, we once again consider the maximum number of edges possible given this construction. Notice $\deg(x) = |X|$, and $\deg(y) = |Y|$. Without loss of generality, assume $Y$ is the set that satisfies Condition 1 above (every vertex in $Y$ has at least one non-edge to $X$). Then $$s(G) \leq |X| + |Y| + |X||Y| - |Y| = |X| +|Y| + \bc{|Y|(|X| - 1)}.$$
To this point, we have assumed nothing about the size of $X$ and $Y$ (other than they must both have at least 2 vertices per Theorem \ref{thm:2privateneighbs}). However, we want to maximize the number of edges possible, which means maximizing the highlighted product $\bc{|Y|(|X| - 1)}$. These two factors have a known sum of $n - 3$, so when $|Y|(|X| - 1) = \lceil\frac{n-3}{2}\rceil \lfloor \frac{n-3}{2}\rfloor$, their product is maximized. Thus, an upper bound on the number of edges possible when $x$ and $y$ live in separate partite sets and are non-adjacent is given by
\begin{align*}
    m_3(G) = n - 2 + \Big\lceil\frac{n-3}{2}\Big\rceil \Big\lfloor \frac{n-3}{2}\Big\rfloor
\end{align*}
Once again, the above equation will have two simplified solutions depending on the parity of $n$. When $n$ is even (thus $n-3$ is odd)
\begin{align*}
    m_3(G) &= n-2 + \frac{n^2-6n + 8}{4}\\
    &= \frac{4n {-8} +n^2-6n + {8}}{4}\\
    m_3(G)&= \frac{n(n-2)}{4}.
\end{align*}
When $n$ is odd (thus $n-3$ is even), we get 
\begin{align*}
    m_3(G) &= n - 2 + \frac{n^2 - 6n + 9}{4}\\
    &= \frac{4n - 8 + n^2 -6n + 9}{4}\\
    m_3(G) &= \frac{(n-1)^2}{4}
\end{align*}
So, $$m_3(G) = \begin{cases}
    \frac{n(n-2)}{4}, \text{     if $n$ is even}\\
        \frac{(n-1)^2}{4},\text{     if $n$ is odd.}
\end{cases}$$
Notice that $m_3(G) = m(G)$. This is no coincidence, as we saw in Section \ref{sect:Tightness}, this is the structure that realizes the bound.\\

We have examined all possible bipartite graph constructions on $n \geq 6$ vertices with $\gamma = 2$ and found none which can achieve more edges than the proposed $m(G)$. This completes the proof.
\end{proof}
\subsection{The Bipartite Bound when $n = 3\gamma$}\label{sect:n3g}
In this section, we show the bipartite bound is true when a graph has exactly $3\gamma$ vertices.
\begin{thm}\label{thm:n=3g}
    Let $G$ be a bipartite graph with a unique minimum dominating set of size $\gamma$. Let $G$ $n = 3\gamma$ vertices. Then the size of $G$ is bounded above by $$m(G) = 2\gamma + 2\Big\lceil\frac{\gamma}{2}\Big\rceil\Big\lfloor\frac{\gamma}{2}\Big\rfloor.$$
\end{thm}

\begin{proof}
     Let $G$ have partite sets $A$ and $B$. Let $D = D_X \cup D_Y$ be the UMD set of $G$. Let $X$ be the set of (non-dominator) vertices dominated by $D_X$ and $Y$  be the set of (non-dominator) vertices dominated by $D_Y$. Finally, let $A = Y \cup D_X$ and $B = X \cup D_Y$. This vertex labeling is visualized below.

\begin{figure}[H]
\centering
   \begin{tikzpicture}
    \node[ellipse,
    draw = black,
    text = black,
    minimum width = 2cm, 
    minimum height = 4cm,label=above:$A$] (e) at (-2,0) {};
    
    \node[ellipse,
    draw = black,
    text = black,
    minimum width = 1cm, 
    minimum height = 2cm] (e) at (-2,1) {$Y$};
     \node[ellipse,
    draw = black,
    text = black,
    minimum width = 1cm, 
    minimum height = 2cm] (e) at (-2,-1) {$D_X$};
    
    \node[ellipse,
    draw = black,
    text = black,
    minimum width = 2cm, 
    minimum height = 4cm,,label=above:$B$] (e) at (2,0) {};
    \node[ellipse,
    draw = black,
    text = black,
    minimum width = 1cm, 
    minimum height = 2cm] (e) at (2,1) {$X$};
     \node[ellipse,
    draw = black,
    text = black,
    minimum width = 1cm, 
    minimum height = 2cm] (e) at (2,-1) {$D_Y$};
\end{tikzpicture}
\caption{$n = 3\gamma$ vertex construction}
\label{fig:3gGraph}
\end{figure}
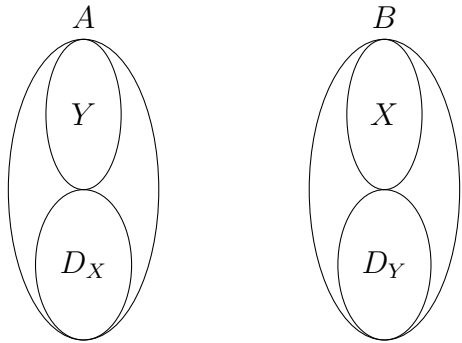
Let $D_X = \{x_1, x_2,\ldots\}$ and $D_Y = \{y_1,y_2,\ldots\}$. By Theorem \ref{thm:2privateneighbs}, each $x$ and $y$ in $D_X$ and $D_Y$ respectively must have two exterior private neighbors with respect to the set $D$. In this case, there are exactly $2\gamma = 2|D|$ vertices in $V(G) - D$. So, the edges between $D$ and $V(G)-D$ must be exactly a 1:2 relation (each vertex in $D$ ``maps" to two vertices in $V(G) - D$). This brings us to our first case.
\proofpart{All dominators in $D$ live in the same partite set.}
Without loss of generality, assume $|D_Y| = 0$, thus $D_X = D$.It follows that $|Y| = 0$. Thus, there are $2\gamma$ vertices in $X$. By the logic above, the only edges that can exist between the two non-empty sets ($D_X$ and $X$) is a 1:2 relation. This results in $s(G) = 2\gamma$.
\proofpart{The dominators of $D$ live in both partite sets.}
To proceed with this proof, we will further articulate the vertex labeling laid out in Figure \ref{fig:3gGraph}. By construction, $|X| = 2|D_X|$ and $|Y| = 2|D_Y|$. So, let $X = \{b_{1,1}, b_{1,2}, b_{2,1},b_{2,2}, \ldots\}$ and $Y = \{a_{1,1}, a_{1,2}, a_{2,1},a_{2,2}, \ldots\}$. Let $\{a_{i,1}, a_{i,2}\} \in N(y_i)$ and $\{b_{j,1}, b_{j,2}\} \in N(x_j)$. Notice that no other edges other than the 1:2 relation described in the previous sentence can exist between $D_X, X$ and $D_Y,Y$. So, we analyze the edges between $X$ and $Y$, as well as the edges between $D_X$ and $D_Y$.\\
Consider the group of vertices $K = \{x_i, b_{i,1},b_{i,2}, y_j,a_{j,1},a_{j,2}\}$. Notice by construction, $D - \{x_i,y_j\}$ dominate $G - K$. So, if any edges we consider cause another duo of vertices (other than $\{x_i,y_j\}$) to dominate $K$, we will find a contradiction. By construction, the following edges already exist:
\begin{figure}[H]
\centering
   \begin{tikzpicture}
   \tikzstyle{vertex}=[circle,fill=black,inner sep=2pt]
    \node[ellipse,
    draw = black,
    text = black,
    minimum width = 2cm, 
    minimum height = 4cm,label=above:$A$] (e) at (-2,0) {};
    
    \node[ellipse,
    draw = black,
    text = black,
    minimum width = 1cm, 
    minimum height = 2cm] (e) at (-2,1) {};
     \node[ellipse,
    draw = black,
    text = black,
    minimum width = 1cm, 
    minimum height = 2cm] (e) at (-2,-1) {};

    \node[ellipse,
    draw = black,
    text = black,
    minimum width = 2cm, 
    minimum height = 4cm,,label=above:$B$] (e) at (2,0) {};
   
    \node[ellipse,
    draw = black,
    text = black,
    minimum width = 1cm, 
    minimum height = 2cm] (e) at (2,1) {};
   
     \node[ellipse,
    draw = black,
    text = black,
    minimum width = 1cm, 
    minimum height = 2cm] (e) at (2,-1) {};

\node [draw, vertex, label = above: $a_{j,1}$] (aj1) at (-2,1.25) {};
\node [draw, vertex, label = below: $a_{j,2}$] (aj2) at (-2,0.75) {};

\node [draw, vertex, label = above: $b_{i,1}$] (bi1) at (2,1.25) {};
\node [draw, vertex, label = below: $b_{i,2}$] (bi2) at (2,0.75) {};

\node [draw, vertex, label = below: $x_i$] (xi) at (-2,-1) {};
\node [draw, vertex, label = below: $y_j$] (yj) at (2,-1) {};

 \draw (bi2) -- (xi) -- (bi1); 
 \draw (aj2) -- (yj) -- (aj1); 
\end{tikzpicture}
\caption{Existing edges by construction}
\label{fig:3gStartingEdges}
\end{figure}
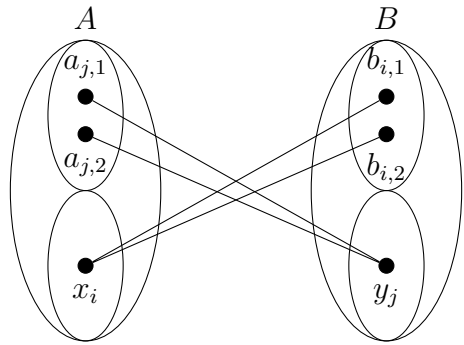
Notice, the $a$ and $b$ vertices are symmetric. So, we will consider two cases:\\
\textbf{Case 1.} $x_i$ and $y_j$ are not adjacent.\\
In this case, there are four possible edges (the complete bipartite graph between the $a$ and $b$ vertices). Without loss of generality, assume $\{a_{j,1}, b_{i,1}\} \in E(G)$. Now, consider what would happen if the edge $\{a_{j,2}, b_{i,2}\}$ existed.
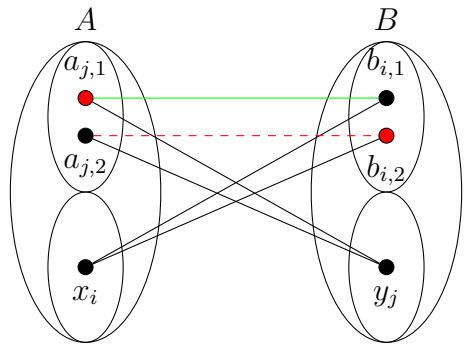
\begin{figure}[H]
\centering
   \begin{tikzpicture}
   \tikzstyle{vertex}=[circle,fill=black,inner sep=2pt]
   \tikzstyle{d}=[circle,fill=red,inner sep=2pt]
    \node[ellipse,
    draw = black,
    text = black,
    minimum width = 2cm, 
    minimum height = 4cm,label=above:$A$] (e) at (-2,0) {};
    
    \node[ellipse,
    draw = black,
    text = black,
    minimum width = 1cm, 
    minimum height = 2cm] (e) at (-2,1) {};
     \node[ellipse,
    draw = black,
    text = black,
    minimum width = 1cm, 
    minimum height = 2cm] (e) at (-2,-1) {};

    \node[ellipse,
    draw = black,
    text = black,
    minimum width = 2cm, 
    minimum height = 4cm,,label=above:$B$] (e) at (2,0) {};
   
    \node[ellipse,
    draw = black,
    text = black,
    minimum width = 1cm, 
    minimum height = 2cm] (e) at (2,1) {};
   
     \node[ellipse,
    draw = black,
    text = black,
    minimum width = 1cm, 
    minimum height = 2cm] (e) at (2,-1) {};

\node [draw, d, label = above: $a_{j,1}$] (aj1) at (-2,1.25) {};
\node [draw, vertex, label = below: $a_{j,2}$] (aj2) at (-2,0.75) {};

\node [draw, vertex, label = above: $b_{i,1}$] (bi1) at (2,1.25) {};
\node [draw, d, label = below: $b_{i,2}$] (bi2) at (2,0.75) {};

\node [draw, vertex, label = below: $x_i$] (xi) at (-2,-1) {};
\node [draw, vertex, label = below: $y_j$] (yj) at (2,-1) {};

 \draw (bi2) -- (xi) -- (bi1); 
 \draw (aj2) -- (yj) -- (aj1); 
 \draw [green] (aj1) edge (bi1);
  \draw [red,dashed] (aj2) edge (bi2);
\end{tikzpicture}
\end{figure}
In this case, $\{a_{j,1}, b_{i,2}\}$ dominates $K$ which is a contradiction. So, the only other two edges that can exist are $\{a_{j,1}, b_{i,2}\}$ and $\{a_{j,2}, b_{i,1}\}$. Notice, adding one of these edges does not result in any new dominating vertices. However, including both edges results in the following.
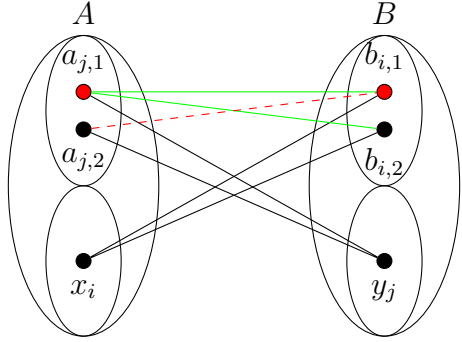
\begin{figure}[H]
\centering
   \begin{tikzpicture}
   \tikzstyle{vertex}=[circle,fill=black,inner sep=2pt]
   \tikzstyle{d}=[circle,fill=red,inner sep=2pt]
    \node[ellipse,
    draw = black,
    text = black,
    minimum width = 2cm, 
    minimum height = 4cm,label=above:$A$] (e) at (-2,0) {};
    
    \node[ellipse,
    draw = black,
    text = black,
    minimum width = 1cm, 
    minimum height = 2cm] (e) at (-2,1) {};
     \node[ellipse,
    draw = black,
    text = black,
    minimum width = 1cm, 
    minimum height = 2cm] (e) at (-2,-1) {};

    \node[ellipse,
    draw = black,
    text = black,
    minimum width = 2cm, 
    minimum height = 4cm,,label=above:$B$] (e) at (2,0) {};
   
    \node[ellipse,
    draw = black,
    text = black,
    minimum width = 1cm, 
    minimum height = 2cm] (e) at (2,1) {};
   
     \node[ellipse,
    draw = black,
    text = black,
    minimum width = 1cm, 
    minimum height = 2cm] (e) at (2,-1) {};

\node [draw, d, label = above: $a_{j,1}$] (aj1) at (-2,1.25) {};
\node [draw, vertex, label = below: $a_{j,2}$] (aj2) at (-2,0.75) {};

\node [draw, d, label = above: $b_{i,1}$] (bi1) at (2,1.25) {};
\node [draw, vertex, label = below: $b_{i,2}$] (bi2) at (2,0.75) {};

\node [draw, vertex, label = below: $x_i$] (xi) at (-2,-1) {};
\node [draw, vertex, label = below: $y_j$] (yj) at (2,-1) {};

 \draw (bi2) -- (xi) -- (bi1); 
 \draw (aj2) -- (yj) -- (aj1); 
 \draw [green] (aj1) edge (bi1);
  \draw [green] (aj1) edge (bi2);
  \draw [red,dashed] (aj2) edge (bi1);
\end{tikzpicture}
\end{figure}
In this case, $\{a_{j,1}, b_{i,1}\}$ dominates $K$, resulting in a contradiction. Thus, in the scenario where $x_i$ and $y_j$ are not adjacent, there can be a maximum of \textbf{two additional edges} ($\{a_{j,1}, b_{i,1}\}$ and $\{a_{j,1}, b_{i,2}\}$, up to symmetry).\\
\textbf{Case 2.} $x_i$ and $y_j$ are adjacent.\\
In this case, there are still four possible edges (the complete bipartite graph between the $a$ and $b$ vertices) to consider. Without loss of generality, assume $\{a_{j,1}, b_{i,1}\} \in E(G)$.
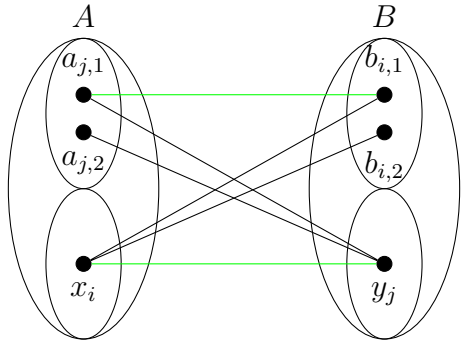
\begin{figure}[H]
\centering
   \begin{tikzpicture}
   \tikzstyle{vertex}=[circle,fill=black,inner sep=2pt]
   \tikzstyle{d}=[circle,fill=red,inner sep=2pt]
    \node[ellipse,
    draw = black,
    text = black,
    minimum width = 2cm, 
    minimum height = 4cm,label=above:$A$] (e) at (-2,0) {};
    
    \node[ellipse,
    draw = black,
    text = black,
    minimum width = 1cm, 
    minimum height = 2cm] (e) at (-2,1) {};
     \node[ellipse,
    draw = black,
    text = black,
    minimum width = 1cm, 
    minimum height = 2cm] (e) at (-2,-1) {};

    \node[ellipse,
    draw = black,
    text = black,
    minimum width = 2cm, 
    minimum height = 4cm,,label=above:$B$] (e) at (2,0) {};
   
    \node[ellipse,
    draw = black,
    text = black,
    minimum width = 1cm, 
    minimum height = 2cm] (e) at (2,1) {};
   
     \node[ellipse,
    draw = black,
    text = black,
    minimum width = 1cm, 
    minimum height = 2cm] (e) at (2,-1) {};

\node [draw, vertex, label = above: $a_{j,1}$] (aj1) at (-2,1.25) {};
\node [draw, vertex, label = below: $a_{j,2}$] (aj2) at (-2,0.75) {};

\node [draw, vertex, label = above: $b_{i,1}$] (bi1) at (2,1.25) {};
\node [draw, vertex, label = below: $b_{i,2}$] (bi2) at (2,0.75) {};

\node [draw, vertex, label = below: $x_i$] (xi) at (-2,-1) {};
\node [draw, vertex, label = below: $y_j$] (yj) at (2,-1) {};

\draw (bi2) -- (xi) -- (bi1); 
\draw (aj2) -- (yj) -- (aj1); 
\draw [green] (aj1) edge (bi1);
\draw [green] (xi) edge (yj);
\end{tikzpicture}
\end{figure}
Similar to above, the edge $\{a_{j,2}, b_{i,2}\}$ would result in a contradiction. The other two edges are symmetric, so without loss of generality, assume the edge $\{a_{j,1}, b_{i,2}\}$ existed.
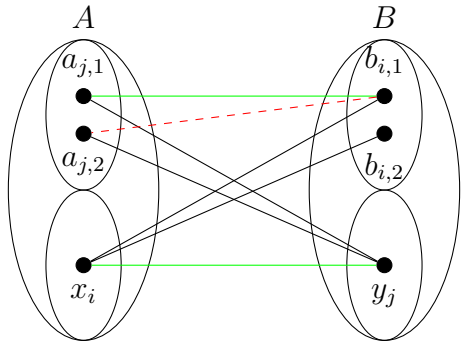
\begin{figure}[H]
\centering
   \begin{tikzpicture}
   \tikzstyle{vertex}=[circle,fill=black,inner sep=2pt]
   \tikzstyle{d}=[circle,fill=red,inner sep=2pt]
    \node[ellipse,
    draw = black,
    text = black,
    minimum width = 2cm, 
    minimum height = 4cm,label=above:$A$] (e) at (-2,0) {};
    
    \node[ellipse,
    draw = black,
    text = black,
    minimum width = 1cm, 
    minimum height = 2cm] (e) at (-2,1) {};
     \node[ellipse,
    draw = black,
    text = black,
    minimum width = 1cm, 
    minimum height = 2cm] (e) at (-2,-1) {};

    \node[ellipse,
    draw = black,
    text = black,
    minimum width = 2cm, 
    minimum height = 4cm,,label=above:$B$] (e) at (2,0) {};
   
    \node[ellipse,
    draw = black,
    text = black,
    minimum width = 1cm, 
    minimum height = 2cm] (e) at (2,1) {};
   
     \node[ellipse,
    draw = black,
    text = black,
    minimum width = 1cm, 
    minimum height = 2cm] (e) at (2,-1) {};

\node [draw, vertex, label = above: $a_{j,1}$] (aj1) at (-2,1.25) {};
\node [draw, vertex, label = below: $a_{j,2}$] (aj2) at (-2,0.75) {};

\node [draw, vertex, label = above: $b_{i,1}$] (bi1) at (2,1.25) {};
\node [draw, vertex, label = below: $b_{i,2}$] (bi2) at (2,0.75) {};

\node [draw, vertex, label = below: $x_i$] (xi) at (-2,-1) {};
\node [draw, vertex, label = below: $y_j$] (yj) at (2,-1) {};

\draw (bi2) -- (xi) -- (bi1); 
\draw (aj2) -- (yj) -- (aj1); 
\draw [green] (aj1) edge (bi1);
\draw [green] (xi) edge (yj);
\draw [red, dashed] (bi1) edge (aj2);
\end{tikzpicture}
\end{figure}
In this case, the set $\{a_{j,1},y_j\}$ dominates $K$, leading to a contradiction. Thus, in the scenario where $x_i$ and $y_j$ are adjacent, there can be a maximum of \textbf{two additional edges} ($\{x_i,y_j\}$ and $\{a_{j,1}, b_{i,1}\}$, up to symmetry).\\
So, we began with $2\gamma$ edges that must exist between $D$ and $G-D$. Now, for every choice of $x \in D_X$ and $y \in D_Y$, there can be at most two additional edges. Thus, $$s(G) \leq 2\gamma + 2(|D_X||D_Y|).$$ Now, $|D_X| + |D_Y| = \gamma$, thus to maximize their product, we want the two sets to be as close to equal in size as possible. This is achieved by $$s(G) \leq2\gamma + 2(|D_X||D_Y|) \leq 2\gamma + 2\Big(\Big\lceil\frac{\gamma}{2} \Big\rceil \Big\lfloor\frac{\gamma}{2} \Big\rfloor\Big)$$
which completes the proof.
\end{proof}
In the next section, we provide our own general graph construction that meets Fischermann's bound.
\subsection{Another Construction for Fischermann's Bound}\label{sect:FishConstruction}
In their 2003 paper, Fischermann et al. provide a generalized construction that meets their proposed edge bound, demonstrating the bound is tight \cite{Original_Paper}. In their construction, the dominating set forms a clique, and many vertices are dominated by more than one vertex. In this section, we provide another construction that meets Fischermann's bound. This time, the construction results in a perfectly dominated graph (Definition \ref{defn:PerfecrDom}). Recall, all graphs with a dominating number of one are perfectly dominated, so a construction is given for $\gamma \geq 2$.
\begin{thm}\label{thm:PerfectGraphAllG}
     Let $\mathcal{G}_{n,\gamma}$ be the set of all uniquely-dominatable simple graphs without isolated vertices with order $n$ and dominating number $\gamma$. Then, for each choice of $n\geq 3\gamma$ and $\gamma \geq 2$, there exists a perfectly dominated $G \in \mathcal{G}_{n,\gamma}$ where $$s(G) = \mathbf{m}(n,\gamma) = {n - \gamma\choose{2}} - \gamma(\gamma-2).$$
\end{thm}
\begin{proof}
    Let $G$ be a simple graph with order $n$. Let $n = 3\gamma + c$. Let the sets $D,A,B,R$ partition $V(G)$. Let $|D| = |A| = |B| = \gamma$, thus $|R| = c$. Let $D = \{x_1,x_2,\ldots,x_\gamma\}$, $A = \{a_1,a_2,\ldots,a_\gamma\}$, $B = \{b_1,b_2,\ldots,b_\gamma\}$, and $R = \{r_1,r_2,\ldots,r_c\}$. Now, add the following edges $G$.
\begin{enumerate}
 \item $\{a_i,x_i\}$ and $\{b_i,x_i\}$ for $i = 1,2,\ldots,\gamma$.
 \item $\{x_1,r\}$ for all $r \in R$.
 \item $\{b_i,a_j\}$ for all combinations of $i = 2,3,\ldots,\gamma$ and $j = 1,2,\ldots,\gamma$ where $i > j$.
 \item $\{b_i,r\}$ for all combinations of $i = 2,3,\ldots,\gamma$ and $r\in R$.
\item All edges possible (a complete subgraph) within the set $A \cup R$.
\end{enumerate}
Now, we must show two things: $s(G) = \mathbf{m}(n,\gamma)$ and $G$ has a unique minimum dominating set.
\proofpart{The size of $G$.}
Let us quantify how many edges each of the five groups of edges enumerated above added to $G$. The subsequent enumeration follows the same order as above.
\begin{enumerate}
    \item 2$\gamma$ edges added.
    \item $r$ edges added.
    \item $\sum_{j=1}^\gamma (\gamma - j) = (\gamma - 1) + (\gamma - 2) + \cdots + 2 + 1 = \frac{(\gamma-1)(\gamma)}{2} = $ ${\gamma}\choose{2}$ edges added.
    \item $r(\gamma-1)$ edges added.
    \item ${\gamma + r}\choose{2}$ edges added.
\end{enumerate}
So, $$s(G) = {{\gamma + r}\choose{2}} + {{\gamma}\choose{2}} + 2\gamma + r(\gamma - 1) + r.$$
Let us expand $s(G)$:
\begin{align*}
    s(G) &= \frac{(\gamma + r)(\gamma + r - 1)}{2} + \frac{\gamma(\gamma - 1)}{2} + \frac{4\gamma}{2} + \frac{2r(\gamma - 1)}{2} + \frac{2r}{2}\\
    &= \frac{1}{2}\bigg(\gamma^2 + r\gamma - \gamma + r\gamma + r^2 - r + \gamma^2 - \gamma +  4\gamma + 2r\gamma \cancel{- 2r} + \cancel{2r} \bigg)\\
    &=\boxed{ \frac{1}{2} \bigg(2\gamma^2 + 4r\gamma + 2\gamma + r^2 - r\bigg)}
\end{align*}
Now, let us expand $\mathbf{m}(n,\gamma)$, plugging in $n = 3\gamma + r$:
\begin{align*}
    \mathbf{m}(n,\gamma) &= {3\gamma + r - \gamma\choose{2}} - \gamma(\gamma-2)\\
    &=\frac{(2\gamma + r)(2\gamma + r-1)}{2} - \frac{2\gamma(\gamma-2)}{2}\\
    &= \frac{1}{2}\bigg(4\gamma^2 + 2r\gamma -2\gamma + 2r\gamma + r^2 - r -2\gamma^2 + 4\gamma \bigg)\\
    &=  \boxed{\frac{1}{2}\bigg(2\gamma^2 + 4r\gamma + 2\gamma + r^2 - r \bigg)}
\end{align*}
It is clear that $s(G) = \mathbf{m}(n,\gamma)$.
\proofpart{$G$ has a unique minimum dominating set.}
Notice, even in this general construction, Observation \ref{obs:NoCommonNeighbs} from the proof of Theorem \ref{thm:Constructions} applies to $G$ here as well. That is to say $D$ is an independent set, and the vertices of $D$ share no common neighbors which implies $\gamma(G) \geq \gamma$. Now, $D$ is a valid dominating set as it clearly dominates $A$, $B$, and $R$. Since $|D| = \gamma$, it follows that $\gamma(G) = \gamma$. So, we show that $D$ is unique. By Observation \ref{obs:NoCommonNeighbs}, if we want to build a dominating set $K = \{k_1,k_2,\ldots,k_\gamma\}$ of size $\gamma$, each subsequent vertex of $K$ must be a vertex in the closed neighborhood of a different vertex of $D$. That is to say $k_1 \in \{a_1,b_1,x_1\} \cup R$ and $k_i \in \{a_i,b_i,x_i\}$ for $i = 2,3,\ldots,\gamma$ (up to relabeling).\\
For the sake of contradiction, assume $k_1 \in \{a_1\} \cup R$. Notice, no vertex in $K$ can dominate $b_1$ (as its only neighbor is $x_1$ which we have ruled out of being in $K$ in this case). So, $k_1 \in \{b_1,x_1\}$. Now, $k_2 \in \{a_2,b_2,x_2\}$. If $k_2 = a_2$, then no vertex in $K$ can dominate $b_2$ (as its only neighbors are $x_2$ and $b_1$, both of which have have been ruled out of $K$ in this case). So, $k_2 \in \{b_2,x_2\}$. Now, $k_3 \in \{a_3,b_3,x_3\}$. Notice, $N(b_3) = \{x_3,a_2,a_1\}$. The latter two vertices have been ruled out of $K$, so once again $k_3 \not=a_3$. Notice, this pattern will continue to unravel (given the symmetry of the chosen edges), and leave us with $k_i \in \{b_i,x_i\}$. Now, $N(a_\gamma) \subset R \cup A \cup \{x_\gamma\}$, the former two sets are already ruled out of $K$. So, $k_\gamma \not= b_\gamma$. Now $N(a_{\gamma - 1}) \subset R\cup A \cup \{b_{\gamma}\} \cup \{x_{\gamma - 1} \}$. The former three sets are already ruled out of $K$, so $k_{\gamma-1} \not =b_{\gamma - 1}$. Once again, this pattern will unravel (this time in the reverse direction) all the way down until $k_i \in \{x_i\}$, which tells us $K = D$ is a unique minimum dominating set.
\end{proof}
As discussed at the end of Section \ref{sect:Tightness}, it seems counterintuitive that a perfectly dominated graph can have the proposed maximum number of edges possible. Perfect domination restricts edges within the dominating set and between the dominating set and the rest of the graph. Yet, for both our bipartite bound and Fischermann's bound for general graphs, an extremal construction exists that also achieves perfect domination.
\newpage
\begin{spacing}{1}
\bibliographystyle{unsrt}
\bibliography{MyRef}
\end{spacing}
\end{document}